\documentclass[12pt]{amsart}
\usepackage[french, british]{babel}
\usepackage{color}
\usepackage[centertags]{amsmath}
\usepackage[margin=1in]{geometry}
\usepackage{amsfonts}
\usepackage{amsthm}
\usepackage{newlfont}
\usepackage{verbatim}
\usepackage{amscd}
\usepackage{amsgen}
\usepackage{amssymb}
\usepackage{bm}
\usepackage{amssymb,amsmath}
\usepackage{mathtools}
\usepackage{amscd}
\usepackage{enumerate}
\usepackage{color}
\usepackage[all]{xy}
\usepackage{graphicx}
\usepackage{tikz}
\usetikzlibrary{arrows.spaced}

\newlength{\defbaselineskip} 
\newcounter{wsk}
\newcounter{wskk}
\theoremstyle{plain}
\newtheorem{pr}{Algorithm}
\theoremstyle{definition} 
\theoremstyle{definition} \newtheorem{ex}{Example}[section] %

\numberwithin{equation}{section}

\DeclareMathOperator{\ini}{in}

\def\ob{\begin{obs}}
\def\kob{\end{obs}}
\def\dow{\begin{proof}}
\def\kdow{\end{proof}}

\def\tw{\begin{thm}}
\def\ktw{\end{thm}}
\def\hip{\begin{con}}
\def\khip{\end{con}}
\def\lem{\begin{lema}}
\def\klem{\end{lema}}
\def\ex{\begin{exm}}
\def\prog{\begin{pr}}
\def\kprog{\end{pr}}
\def\wn{\begin{cor}}
\def\kwn{\end{cor}}{}
\def\uwa{\begin{rem}}
\def\kuwa{\end{rem}}
\def\kex{\end{exm}}
\def\dfi{\begin{df}}
\def\kdfi{\end{df}}
\def\Tm{{\mathcal T}}

\definecolor{zielony}{rgb}{0.2, 0.5, 0}
\definecolor{czerwony}{rgb}{0.9, 0.2, 0.1}
\definecolor{brazowy}{rgb}{0.5, 0.1, 0.0}
\definecolor{niebieski}{rgb}{0.3, 0.1, 0.9}
\newcommand{\red}[1]{\color{czerwony} #1 \color{black}}

\newenvironment{redd}{\red}{}
\newcommand{\bred}{\begin{redd}}
\newcommand{\ered}{\end{redd}}

\xyoption{line}

\usepackage{hyperref}

\newtheorem{theorem}{Theorem}[section]
\newtheorem*{theorem*}{Theorem}

\newtheorem{lemma}[theorem]{Lemma}
\newtheorem{corollary}[theorem]{Corollary}
\newtheorem{proposition}[theorem]{Proposition}

\newtheorem{con}[theorem]{Conjecture}

\theoremstyle{definition}

\newtheorem*{definition*}{Definition}

\newtheorem{remark}[theorem]{Remark}
\newtheorem{example}[theorem]{Example}

\newtheorem*{notation*}{Notation}

\usepackage{tikz-cd}
\usetikzlibrary{decorations.pathmorphing}

\newcommand{\Zz}{\mathbb{Z}}

\newcommand{\Rr}{\mathbb{R}}
\newcommand{\RR}{\mathbb{R}}

\DeclareMathOperator{\conv}{conv}

\begin{document}
\title[Symmetric edge polytopes]{Arithmetic aspects of symmetric edge polytopes}

\author[A.~Higashitani]{Akihiro Higashitani}
\address{Department of Mathematics\\Graduate School of Science\\Kyoto Sangyo University\\Kamigamo Motoyama\\Kita-ku\\Kyoto\\603-8555\\Japan}
\email{ahigashi@cc.kyoto-su.ac.jp}
\author[K.~Jochemko]{Katharina Jochemko}
\address{Department of Mathematics, Royal Institute of Technology (KTH), SE-100 44 Stockholm, Sweden}
\email{jochemko@kth.se}
\author[M.~Micha{\l}ek]{Mateusz Micha{\l}ek}
 \address{Mateusz Micha{\l}ek\\
 Max Planck Institute for Mathematics in the Sciences, Inselstr. 22\\ 04103 Leipzig, Germany\\
 and
 Institute of Mathematics of the
 Polish Academy of Sciences\\
 ul. \'Sniadeckich 8\\
 00-656 Warszawa, Poland}
 \email{Mateusz.Michalek@mis.mpg.de}

\keywords{Symmetric edge polytope; $h^*$-polynomial; real roots; complete bipartite graph; unimodular triangulation}
\subjclass[2010]{05A15, 52B12 (primary); 13P10, 26C10, 52B15, 52B20 (secondary)}



\selectlanguage{british}
\begin{abstract}
We investigate arithmetic, geometric and combinatorial properties of symmetric edge polytopes. We give a complete combinatorial description of their facets. By combining Gr\"obner basis techniques, half-open decompositions and methods for interlacings polynomials we provide an explicit formula for the $h^\ast$-polynomial in case of complete bipartite graphs. In particular, we show that the $h^\ast$-polynomial is $\gamma$-positive and real-rooted. This proves Gal's conjecture for arbitrary flag unimodular triangulations in this case, and, beyond that, we prove a strengthing due to Nevo and Petersen (2011).
\end{abstract}
\selectlanguage{british}
\maketitle

\section{Introduction}
The investigation of graphs in combination with polytopes has a long tradition. 
Graphs arising from polytopes are a classical, intensively studied topic in discrete geometry and optimization and remain an active area of research (see, e.g.,~\cite{blind1987puzzles,Dantzig,Santos,Steinitz}). On the other hand, a variety of polytope constructions arising from graphs have led to interesting examples and new insights in combinatorics, graph theory, geometry and algebra (see, e.g.,~\cite{deza2009geometry,lovasz1972normal}). 
An important class of examples constitute~\textbf{edge polytopes} which were introduced by Hibi and Ohsugi \cite{ohsugi1998normal}. Of current particular interest due to their intimate relation to matroid polytopes and generalized permutahedra are root polytopes, a reincarnation of edge polytopes for bipartite graphs, introduced by Postnikov~\cite{postnikov2009permutohedra}. For further reading on edge polytopes we refer to~\cite{Dupont,hibi2013separating,Tran}. 

The focus of the present article is \textbf{symmetric edge polytopes}, a symmetrized version of edge polytopes that were introduced in \cite{MHNOH}. We study fundamental geometric and arithmetic properties of symmetric edge polytopes by combining geometric, algebraic, combinatorial and analytic methods of recent special interest: \textbf{half-open decomposition}, \textbf{Gr\"obner bases} and \textbf{interlacing polynomials}. For a simple graph $G$, that is, without loops or multiple edges, with vertex set $V$ and edge set $E$, the symmetric edge polytope is defined as 
\[
P_G:=\mathrm{conv}(e_v-e_w, e_w-e_v : vw \in E) \subset \Rr^V
\]
where $e_v \in \Rr^V$ is the unit vector indexed by the vertex $v$ of $G$. By definition, $P_G$ is a centrally symmetric \textbf{lattice polytope}. Moreover, $P_G$ belongs to the class of \textbf{reflexive} and \textbf{terminal} polytopes~\cite{Higashitani2015} that play a prominent role in algebraic geometry via mirror symmetry~\cite{Batyrev}. A fundamental arithmetic invariant of a lattice polytope is the \textbf{Ehrhart polynomial} which encodes the number of lattice points in its integer dilates~\cite{ehrhartRational}. A fundamental question in Ehrhart theory is to characterize Ehrhart polynomials. Of current particular interest are roots of Ehrhart polynomials and their closely related $h^\ast$-polynomials. Using orthogonal polynomial techniques, it was recently proved that the Ehrhart polynomial of the symmetric edge polytope for complete bipartite graphs $K_{a,b}$ exhibits behavior similar to the Riemann $\zeta$-function whenever $a\leq 3$~\cite{Higashitani2017interlacing} extending investigations initiated by Bump et. al.~\cite{BCKV}. More precisely, all roots lie on the complex line $\{z\in \mathbb{C}\colon \Re (z) =-\frac{1}{2}\}$, where $\Re(z)$ denotes the real part of $z$, and moreover the roots interlace on that line. Interlacing polynomials currently receive considerable attention due to their significance in a recent proof of the Kadison-Singer Problem~\cite{intII}. In the light of Stanley's Unimodality conjecture \cite{stanley1989log} of current great interest in Ehrhart theory are real-rooted $h^\ast$-polynomials and here recent success was made using interlacing polynomials techniques~\cite{beck2016h,jochemko2016real,savage2015,solus2017simplices}. Combining algebraic, geometric and combinatorial counting arguments, we obtain our first main theorem --- the following simple description for $h^\ast$-polynomials of symmetric edge polytopes of complete bipartite graphs.

\begin{theorem*}
For all $a,b\geq 0$ let $h^\ast _{a,b}(t)$ denote the $h^\ast$-polynomial of $P_{K_{a+1,b+1}}$. Then
\begin{equation*}\label{eq:K_ab1}
h^\ast _{a,b}(t) \ = \ \sum_{i=0}^{\min(a,b)}\binom{2i}{i}\binom{a}{i}\binom{b}{i}t^i(1+t)^{a+b+1-2i}\, .
\end{equation*}
\end{theorem*}
Our result generalizes \cite[Proposition 4.4]{Higashitani2017interlacing}, where the case $\min(a,b) \leq 3$ was studied. 
From the formula it is apparent that the $h^\ast$-polynomial is palindromic, that is, $t^{a+b+1}h^\ast _{a,b}(\frac{1}{t})=h^\ast _{a,b}(t)$ which reflects that $P_{K_{a+1,b+1}}$ is reflexive by a famous theorem of Hibi~\cite{Hibidual}. Stronger, our theorem shows that $h^\ast _{a,b}(t)$ is $\gamma$-positive and therefore unimodal. Since every triangulation of the symmetric edge polytope using all lattice points is unimodular~\cite{MHNOH}, the $h^\ast$-polynomial agrees with the $h$-polynomial of any such triangulation. This directly relates to open question in topological combinatorics: since $h^\ast _{a,b}(t)$ is $\gamma$-positive, we answer \textit{Gal's conjecture} \cite[Conjecture 2.1.7]{gal2005real} in the affermative for all flag triangulations of $\partial P_{K_{a+1,b+1}}$, that is, in particular, for the triangulation obtained from the explicit description of the Gr\"obner basis (Theorem~\ref{thm:groebnerKab}). In~\cite{nevo2011gamma} Nevo and Petersen moreover conjecture that the $\gamma$-vector of any flag simplicial sphere is the $f$-vector of a balanced simplicial complex, or equivalently, that its entries satisfy the Frankl--F\"uredi--Kalai inequalities. In Section~\ref{sec:Nevo} we confirm that for all flag unimodular triangulations of $\partial P_{K_{a+1,b+1}}$ by giving an explicit construction of a corresponding simplicial complex. Using $\gamma$-positivity and a classical result of Poly\'a and Schur~\cite{polya1914zwei} we furthermore prove our second main theorem; $h^\ast _{a,b}(t)$ has only real roots and the following interlacing property.
\begin{theorem*}\label{thm:main2}
For all $a,b\geq 0$ the polynomial $h^\ast _{a,b}(t)$ has only real roots and
\[
h^\ast _{a,b-1}(t) \preceq h^\ast_{a,b}(t) \, .
\]
That is, the polynomial $h^\ast _{a,b-1}(t)$ interlaces $h^\ast_{a,b}(t)$.
\end{theorem*}

The outline of the paper is as follows: In Section~\ref{sec:prelim} we provide preliminaries and notation of our main objects. In Theorem~\ref{thm:facets} we give a combinatorial description 
of all facet defining hyperplanes of $P_G$. From that we derive a simple counting formula for the number of facets in case of bipartite (Proposition~\ref{prop:facetsbipartite}) and, more generally, multipartite graphs (Proposition \ref{prop:facetskpartite}) which is exponential in the number of vertices. In Section~\ref{sec:Groebner} we give a combinatorial description of a Gr\"obner basis of $P_G$. Section~\ref{sec:bipartite} is devoted to complete bipartite graphs collecting all ingredients of the proofs of our two main theorems. In Section~\ref{sec:goodcolorings} we provide an alternative, combinatorial formula of $h^\ast _{a,b}(t)$ by careful double counting. In Section~\ref{sec:half} we construct a half-open triangulation and obtain a graph theoretical description of $h^\ast _{a,b}(t)$. Section~\ref{sec:proof} is devoted to proving the first main theorem (Theorem \ref{thm:main1}). 
In Section~\ref{sec:roots} we study the roots of $h^\ast _{a,b}(t)$ and prove the second main theorem (Theorem \ref{thm:main2}). In Section \ref{sec:Nevo} we prove Conjecture~\ref{conj:nevo} for triangulations of $\partial P_{K_{a,b}}$. We conclude with a general recursive formula for $h^\ast _{a,b}(t)$ in Section~\ref{sec:recursion} .

\section{Preliminaries}\label{sec:prelim}

In the sequel we collect preliminaries necessary for the following sections. We assume basic knowledge of polyhedral geometry and commutative algebra. For further reading we recommend~\cite{BeckRobins,Braenden,MTor,Stks,ziegler}. 

\subsection{Lattice polytopes} 
A \textbf{lattice polytope} is the convex hull of finitely many elements in a lattice contained in $\mathbb{R}^d$, typically $\mathbb{Z}^d$. A lattice polytope $P$ is called \textbf{reflexive} if
$$P^\vee := \{ {\mathbf u} \in \Rr^d : \langle {\mathbf u}, {\mathbf v}\rangle \geq -1 \text{ for any }{\mathbf v} \in P\}$$
is also a lattice polytope, where $\langle \cdot, \cdot \rangle$ denotes the usual inner product of $\Rr^d$. It is called \textbf{terminal} if all lattice points on the boundary of $P$ are vertices. In particular, the only lattice points that are contained in a terminal reflexive lattice polytope are its vertices and the origin. 

By a theorem of Ehrhart~\cite{ehrhartRational}, $|nP\cap \mathbb{Z}^d|$ is given by a polynomial $E_P(n)$ of degree $\dim P$ in $n$ for all integers $n\geq 0$, the \textbf{Ehrhart polynomial}. The \textbf{$h^\ast$-polynomial} $h^\ast _P (t)=h_0^\ast +h_1 ^\ast t+\cdots +h_dt^d$ of a $d$-dimensional lattice polytope $P$ encodes the Ehrhart polynomial in a particular basis consisting of binomial coefficients:
\[
E_P (n)=h_0 ^\ast {n+d\choose d}+h_1 ^\ast {n+d-1\choose d} +\cdots + h_d^\ast {n\choose d} \, .
\]
A fundamental theorem of Stanley \cite{stanleydecomp} states that the coefficients of the $h^\ast$-polynomial are always nonnegative integers. It was proved by Hibi~\cite{Hibidual} that a $d$-dimensional lattice polytope $P$ is reflexive if and only if its $h^\ast$-polynomial is palindromic, that is, $h^\ast _P (t)=t^dh^\ast _P \left(\frac{1}{t}\right)$.

\subsection{Triangulations and Gr\"obner bases}
A \textbf{triangulation} $\mathcal{T}$ of dimension $d$ is a subdivision into simplices of dimension at most $d$. 
The triangulation $\mathcal{T}$ is \textbf{flag} if every minimal non-face of $\mathcal{T}$ is $1$-dimensional. If $\mathcal{T}$ has vertex set $V$, then $\mathcal{T}$ is \textbf{balanced} if there is a proper coloring of its vertices $c:V \rightarrow [d+1]$, i.e., for every face $F \in \mathcal{T}$, the restriction of $c$ into $F$ is injective. 
The \textbf{$f$-polynomial} $f(t)=f_{-1}+f_0t+\cdots +f_dt^{d+1}$ encodes the numbers of faces in all dimensions: $f_i=|\{\Delta \in \mathcal{T}\colon \dim \Delta =i\}|$, where we let $f_{-1}:=1$. The \textbf{$h$-polynomial} $h(t)=\sum_{i=0}^{d+1} h_it^i$ is given via the following relation:
\[
f(t)=\sum _{i=0}^{d+1} h_i t^i(1+t)^{d+1-i}. 
\]

Note that $h_{d+1}=0$ if the geometric realization of $\mathcal{T}$ is homeomorphic to the $d$-dimensional ball, so in our case we always have $h_{d+1}=0$.  
A $d$-dimensional lattice simplex is called \textbf{unimodular} if its vertices affinely span the integer lattice $\mathbb{Z}^d$. A triangulation of a lattice polytope into unimodular simplices is called a \textbf{unimodular triangulation}. If $\mathcal{T}$ is a unimodular triangulation of a lattice polytope $P$ then its $h$-polynomial equals the $h^\ast$-polynomial $h^\ast_P(t)$ of $P$, and moreover, if $P$ is reflexive, then it is equal to the $h$-polynomial of the induced unimodular triangulation of the boundary.

An important tool to calculate triangulations are Gr\"obner bases. Let $K$ be a field and let $K[t_1^\pm, \ldots,t_d^\pm,s]$ denote the ring of Laurent polynomials in $(d+1)$ variables. Let $P$ be a fixed lattice polytope in $\mathbb{R}^d$. For any lattice point $\alpha =( \alpha_1,\ldots,\alpha_d) \in P \cap \Zz^d$, 
let $u_\alpha$ be the Laurent monomial $t_1^{\alpha_1}\cdots t_d^{\alpha_d} \in K[t_1^\pm, \ldots,t_d^\pm,s]$. The \textbf{toric ring} $K[P]$ of $P$ is the subring of $K[t_1^\pm,\ldots,t_d^\pm,s]$ generated by those monomials $u_\alpha s$ with $\alpha \in P \cap \Zz^d$. Let $S=K[x_\alpha : \alpha \in P \cap \Zz^d]$ be the polynomial ring with $|P \cap \Zz^d|$ variables and $\deg(x_\alpha)=1$. 
Then $\pi : S \rightarrow K[P], \pi(x_\alpha)=u_\alpha s$ defines a surjective ring homomorphism. The kernel of $\pi$ is called the \textbf{toric ideal} of $P$ and is denoted by $I_P$. 

A total order $<$ on the monomials of a polynomial ring is called a \textbf{monomial order} if for all monomials $a,b,c$, one has $ac<bc$ whenever $a<b$ and $1<a$ for all non-constant monomials. An important example is the \textbf{degree reverse lexicographic order $<_\text{rev}$ (degrevlex)}. Here, for two monomials $\prod x_i^{a_i}$ and $\prod x_i^{b_i}$, $\prod x_i^{a_i} <_\text{rev} \prod x_i^{b_i}$ holds with respect to the degree reverse lexicographic order $<_\text{rev}$ induced by the ordering $x_1 <_\text{rev} x_2 <_\text{rev} \cdots $ of variables if and only if $\sum a_i < \sum b_i$, or $\sum a_i = \sum b_i$ and $a_j > b_j$ for $j=\min\{i : a_i \neq b_i\}$.
The \textbf{initial term} of a polynomial $f$ is the largest monomial that appears in $f$ and is denoted by $\ini_< (f)$. The ideal generated by all initial terms of an ideal $I$ is called the \textbf{initial ideal} of $I$ 
and denoted by $\ini _< (I)$. A system of generators $g_1,\ldots,g_m$ of an ideal is called a \textbf{Gr\"obner basis} if the initial terms of the generators already generate the initial ideal of $I$, that is, if $\ini_< (I) =\langle \ini _< (g_1),\ldots,\ini_<(g_m) \rangle$. 

Consider any set of polynomials $F=\{f_1,\dots,f_k\}$ with leading monomials $m_1,\dots,m_k$. Suppose the coefficient of $m_i$ in $f_i$ equals one. We say that a polynomial $P$ may be \textbf{reduced} using $F$, if some monomial $m$ in the support of $P$ is divisible by one of the $m_i$'s, say $m=m_i\cdot m'$. A \textbf{reduction} of $P$ is a polynomial $P'$ with $m$ replaced by $m'(m_i-f_i)$. As every monomial order is a well-ordering, every reduction process must terminate, possibly with zero. A set of polynomials of an ideal $I$ is a Gr\"obner basis if and only if any polynomial $P\in I$ may be reduced to $0$, or equivalently if $P\neq 0$ then $P$ may be reduced.

For a toric ideal $I_P$ of a lattice polytope $P$, let $\Delta_P$ be the collection of subsets $S \subset P \cap \Zz^d$ such that $\mathrm{conv}(S)$ is a simplex and $\prod_{\alpha \in S}x_\alpha \not\in \sqrt{\ini_<(I_P)}$.
Then $\Delta_P$ defines a regular triangulation of $P$ with the vertex set $P \cap \Zz^d$. Moreover, the triangulation $\Delta_P$ is unimodular 
if and only if $\ini_<(I)=\sqrt{\ini_<(I)}$ (see, e.g., \cite[Corollary 8.9]{Stks}). 
In other words, $P$ has a regular unimodular triangulation if and only if $I_P$ has a squarefree Gr\"obner basis, 
where a Gr\"obner basis $g_1,\ldots,g_m$ is said to be \textbf{squarefree} if all its initial terms $\ini _< (g_1),\ldots,\ini_<(g_m)$ are squarefree.

\subsection{Half-open decompositions}
Let $P$ be a full-dimensional polytope with facets $F_1,\ldots, F_m$ and let $q$ be in \textbf{general} position, that is, $q$ is not contained in any facet defining hyperplane for all $i$. A facet $F_i$ of $P$ is \textbf{visible} from $q$ if for every $p\in F_i$ we have $(p,q]\cap P = \emptyset$. Let $I_q (P)=\{i\in [m]\colon F_i \text{ visible}\}$ be the index set of visible facets. A \textbf{half-open polytope} is a polytope without its visible facets:
\[
H_q P \ = \ P\setminus \bigcup _{i\in I_q (P)} F_i \, .
\]
If $P=P_1\cup \cdots \cup P_k$ defines a polyhedral subdivision with maximal cells $P_1,\ldots,P_k$ and $q$ is in general position with respect to all $P_i$ then
\[
H_q P \ = \ H_q P_1 \sqcup \cdots \sqcup H_q P_k
\]
defines a partition~\cite{koppeverdoolaege}. If $P_1,\ldots, P_k$ are simplices, that is, they are the maximal cells of a triangulation $\mathcal{T}$ of $P$ then the $h$-polynomial of the triangulation can be read-off from the half-open decomposition.
\begin{proposition}\label{prop:halfopendecomp}
Let $P=P_1\cup \cdots \cup P_k$ be a triangulation $\mathcal{T}$ and $q$ general with respect to $P_i$ for all $i$. Let $h(t)=h_0+h_1t+\cdots +h_d t^d$ be the $h$-polynomial of $\mathcal{T}$. Then
\[
h_i \ = \ |\{j\in [k]\colon |I_q (P_j)|=i\}| \, .
\]
\end{proposition}
In particular, if $P$ is a lattice polytope and $P=P_1\cup \cdots \cup P_k$ defines a unimodular triangulation, then also $h_i^\ast (P) = |\{j\in [k]\colon |I_q (P_j)|=i\}|$ for all $i$.

\subsection{Real-rooted polynomials}
A polynomial $f=\sum _{i=0} ^d a_i t^i$ of degree $d$ with real coefficients is said to be \textbf{real-rooted}, if all its roots are real. If all coefficients of a real-rooted polynomial are nonnegative, or equivalently, all roots are nonpositive, then $a_i^2\geq a_{i-1}a_{i+1}$ for all $i$~\cite{stanley1989log}. A sequence $a_i$ of coefficients satisfying this system of inequalities is called \textbf{log-concave}. An immediate consequence is that the nonnegative, log-concave sequence is \textbf{unimodal}, that is, $a_0\leq a_1\leq \cdots \leq a_k \geq \cdots \geq a_d$ for some $k$. 

The polynomial $f$ is said to be \textbf{palindromic} if $f(t)=t^d f\left(\frac{1}{t}\right)$. It is \textbf{$\gamma$-positive} if there are $\gamma _0,\gamma _1,\ldots,\gamma _{\lfloor \frac{d}{2}\rfloor}\geq 0$ such that $f(t)=\sum _{i\geq 0}\gamma _i t^i (1+t)^{d-2i}$. The polynomial $\sum \gamma_it^i$ is called the \textbf{$\gamma$-polynomial} of $f$. It can be seen that a $\gamma$-positive polynomial is real-rooted if and only if its $\gamma$-polynomial has only real roots.

Let $f$ and $g$ be real-rooted polynomials with roots $a_1\geq a_2\geq \cdots$, respectively, $b_1\geq b_2\geq \cdots$. Then $g$ is said to \textbf{interlace} $f$ if
\[
a_1\geq b_1 \geq a_2\geq b_2 \geq \cdots \, .
\]
In this case we write $g\preceq f$. In particular, $\deg g\leq \deg f\leq \deg g +1$. If $f\preceq g$ or $g\preceq f$ we say that $f$ and $g$ interlace. By the intermediate value theorem, it follows that $f'\preceq f$ for every real-rooted polynomial $f$. The following result by Obreschkoff~\cite{obreschkoff1963verteilung} characterizes the interlacing property.
\begin{theorem}[\cite{obreschkoff1963verteilung}]
Let $f,g\in \RR [t]$ be polynomials with $|\deg f - \deg g|\leq 1$. Then $f$ and $g$ interlace if and only if $cf+dg$ has only real-roots for all $c,d\in \RR$.
\end{theorem}
A linear operator $T\colon \RR[t] \rightarrow \RR[t]$ preserves real-rootedness if $T(f)$ has only real roots for any real-rooted polynomial $f\in \RR[t]$. An operator is said to preserve the interlacing property if $T(f)$ and $T(g)$ are interlacing whenever $f$ and $g$ are. An immediate consequence of Obreschkoff's theorem is the following.
\begin{corollary}\label{cor:preserving interlacing}
Let $T\colon \RR[t] \rightarrow \RR[t]$ be a linear operator. Then $T$ preserves real-rootedness if and only if it preserves the interlacing property.
\end{corollary}
An operator acts diagonally if there is a sequence $\{ \lambda_i \} _i\geq 0$ such that $T(t^i)=\lambda_i t^i$ for all $i\geq 0$. If $T$ preserves real-rootedness, then $\{ \lambda_i \} _i\geq 0$ is called a \textbf{multiplier sequence}. The following famous theorem by Poly\'a and Schur~\cite{polya1914zwei} characterizes multiplier sequences.
\begin{theorem}[\cite{polya1914zwei}]\label{thm:multiplier}
Let $\Lambda = \{ \lambda_i \} _i\geq 0$ be a sequence and 
\[
G_\Lambda (x) \ := \ \sum _{i\geq 0} \lambda _i \frac{x^i}{i!} \, .
\]
The following are equivalent.
\begin{itemize}
\item[(i)] $\Lambda$ is a multiplier sequence.
\item[(ii)] $G_\Lambda (x)$ is an entire function that is the limit of real-rooted polynomials whose zeros all have the same sign that converge uniformly  on  compact
subsets of $\mathbb{C}$.
\end{itemize}
\end{theorem}

\section{Facets and triangulations}
\subsection{Facets}\label{sec:facets}
In this section we provide a combinatorial description of the facets of the symmetric edge polytope $P_G$ of an arbitrary finite simple 
graph $G=(V,E)$ with vertex set $V$ and edge set $E$. From the definition it follows that the vertices of $P_G$ are contained in the lattice $M = \{x\in \Zz^V\colon \sum _{v\in V} x_v =0\}$. The dual lattice $N= \Zz^{V}/(1_v)_{v\in V}\Zz$ consists of functions $f:V\rightarrow \Zz$, where two functions are identified if they differ by a common constant. Every such function can be identified with associations of integers to vertices of the graph $G$, up to addition of a (common) constant. Since $P_G$ is reflexive, every facet defining hyperplane is of the form $\{x\in M \colon \sum _{v\in V} f(v)x_v=1\}$ for some $f\in N$.

\begin{theorem}\label{thm:facets}
Let $G=(V,E)$ be a finite simple connected graph. Then $f\colon V\rightarrow \Zz$ is facet defining if and only if
\begin{itemize}
\item[(i)] for any edge $e=uv$ we have $|f(u)-f(v)|\leq 1$, and
\item[(ii)] the subset of edges $E_f=\{e=uv\in E \colon |f(u)-f(v)|=1\}$ forms a spanning subgraph of $G$.
\end{itemize}
\end{theorem}
\begin{proof}
First we show that for any function $f\colon V\rightarrow \Zz$ that satisfies the conditions (i) and (ii) the hyperplane $\{x\in M \colon \sum _{v\in V} f(v)x_v=1\}$ defines a facet. 
By condition (ii) the set $X_f=\{e_u-e_v \colon f(u)-f(v)=1\}$ of vertices of $P_G$ that lie on the hyperplane defined by $f$ span $\{x\in \Zz^V\colon \sum _{v\in V} x_v =0\}$. 
Moreover, by condition (i), $f$ is maximized on $X_f$. Therefore, $X_f$ is a facet and thus $f$ is facet defining.

For the other direction, let $F=\{x\in P_G\colon \sum _{v\in V} f(v)x_v=1\}$ be a facet of $P_G$ defined by some $f\colon V \rightarrow \Zz$. Then $e_u-e_v\in F$ if and only if $f(u)-f(v)=1$. Since $F$ is a facet, there are $|V|-1$ linearly independent vertices in $F$, say $e_{u_1}-e_{v_1},\ldots,e_{u_{|V|-1}}-e_{v_{|V|-1}}$. Because of linear independence $\{u_1v_1,\ldots,u_{|V|-1}v_{|V|-1}\}\subset E_f$ is a spanning tree and therefore also $E_f$ is spanning and condition (ii) is satisfied. Since $P_G$ is symmetric, $-F=\{x\in P_G\colon \sum _{v\in V} f(v)x_v=-1\}$ is also a facet and therefore $f(x)\in [-1,1]$ for all $x\in P_G$. In particular, evaluating $f$ at vertices of $P_G$ shows that condition (i) is satisfied.
\end{proof}
From the proof of Theorem~\ref{thm:facets} we see the following.
\begin{corollary}
The unimodular simplices contained in a facet of $P_G$ represented by a function $f$ correspond exactly to spanning trees consisting of all edges $vw$ such that $|f(v)-f(w)|=1$.
\end{corollary}
In case of complete graphs, complete bipartite graphs and, more generally, complete multipartite graphs, Theorem~\ref{thm:facets} leads to a simple description of the facets of the symmetric edge polytope that moreover allows for an easy counting formula. For complete graphs the following is immediate.
\begin{corollary}
Let $K_n$ be a complete graph with vertex set $V$. Then a function $f\colon V\rightarrow \mathbb{Z}$ is facet defining if and only if, up to a constant, $f(V)=\{0,1\}$.
\end{corollary}

\begin{proposition}\label{prop:facetsbipartite}
Let $K_{a,b}$ be a complete bipartite graph with vertex set $v_1,\dots,v_a,w_1,\dots,w_b$ and edge set $\{\{v_i,w_j\} : 1\leq i\leq a, \; 1\leq j\leq b\}$. Then $f:V\rightarrow\Zz$ defines a facet if and only if $f$, up to a constant, satisfies one of the following conditions.
\begin{enumerate}
\item[(i)] $f(v_i)=0$ for all $i$ and $f(w_j)\in \{-1,1\}$ for all $j$, or
\item[(ii)] $f(w_j)=0$ for all $j$ and $f(v_i)\in \{-1,1\}$ for all $i$.
\end{enumerate}
In particular, the polytope $P_{K_{a,b}}$ has  $2^a+2^b-2$ facets.
\end{proposition}
\begin{proof}
By Theorem~\ref{thm:facets}, every function $f$ that satisfies condition (i) or (ii) defines a facet. It thus remains to prove that every $f$ such that $F=\{x\in P_G\colon \sum _{v\in V} f(v)x_v=1\}$ is a facet satisfies, up to  a constant, one of the conditions (i) or (ii). First, we assume that $f$ is constant on one part of the graph; without loss of generality $f(v_i)=0$ for all $1\leq i\leq a$. By condition (i) in Theorem~\ref{thm:facets}, $f(w_j)\in \{-1,0,1\}$ for all $1\leq j\leq b$. However, if there was a $j$ with $f(w_j)=0$, then the graph given in (ii) of Theorem~\ref{thm:facets} was not connected. Hence we must have $f(w_j)\in \{-1,1\}$ for all $1\leq j\leq b$.

We are left with excluding the possibility that $f$ is nonconstant on both parts of the graph. We give a proof by contradiction and assume that there is such a facet defining function $f$. Without loss of generality we may assume that $f(v_1)=0$, $f(v_i)\geq 0$ for all $2\leq i\leq a$ and there exists a $2\leq k\leq a$ with $f(v_k)>0$. If $f(v_k)>1$ then $f(w_j)=1$ must hold for all $j$ in order to be able to satisfy condition (i) of Theorem~\ref{thm:facets}. This, however, is a contradiction to the assumption that $f$ is non-constant on both parts. Thus, $f(v_i)\in \{0,1\}$ for all $2\leq i\leq a$ and hence $f(w_j)\in \{0,1\}$ for all $j$. However, then the graph given in (ii) of Theorem~\ref{thm:facets} is not connected, again a contradiction.
\end{proof}

\begin{proposition}\label{prop:facetskpartite}
Let $k\geq 3$ and $G=K_{a_1,\dots,a_k}$ be a complete $k$-partite graph with vertex set $V=\bigcup _{i=1}^k A_i$ and edge set $\{uv \colon u\in A_i, v\in A_j, i\neq j\}$. Then $f\colon V\rightarrow \mathbb{Z}$ is facet defining if and only if $f$, up to a constant, satiesfies one of the following conditions.
\begin{enumerate}
\item[(i)] $f(A_i)=\{-1,1\}$ for some $1\leq i\leq k$ and $f_{|A_j}=0$ for all $i\neq j$, or 
\item[(ii)] $f(V)=\{0,1\}$ and
\subitem[a] $f$ is constant on $A_i$ for all $1\leq i\leq k$, or
\subitem[b] there exist an $i\neq j$ such that $f(A_i)=\{0,1\}=f(A_j)$ . 
\end{enumerate}
In particular, the polytope $P_G$ has 
$2^{\sum_{i=1}^k a_i}-\sum_{i=1}^k (2^{a_i}-2)-2$
facets. 
\end{proposition}
\begin{proof}
It is easy to check that any function satisfying the above conditions also satisfies the conditions given in Theorem~\ref{thm:facets} and is thus facet defining.

For the other direction, let $f\colon V\rightarrow \mathbb{Z}$ be a facet defining function. Without loss of generality we may assume that $f(v)=0$ for some vertex $v$ in $A_1$. Then by condition (i) of Theorem~\ref{thm:facets} $f(u)\in [-1,1]$ for all $u\in V\setminus A_1$. 

First suppose that there exists an $i>1$ with $\{-1,1\}\subseteq f(A_i)$. Then, by condition (i) of Theorem~\ref{thm:facets}, $f_{|A_j}=0$ for all $j\neq i$ and then, by condition (ii) of Theorem~\ref{thm:facets}, $f(A_i)=\{-1,1\}$ and thus condition (i) above is satisfied.

Otherwise, without loss of generality we may assume $f$ takes only values $0$ and $1$, as it is not possible that $f$ takes the value $-1$ on one part and $1$ on another by condition (i) of Theorem~\ref{thm:facets}. If $f(A_i)=\{0,1\}=f(A_j)$ for some $i\neq j$ then condition (ii)[b] above is satisfied. Since, by condition (ii) of Theorem~\ref{thm:facets}, it is not possible that $f$ is non-constant on one part $A_i$ and constant and equal on all other parts, we proved the claim.

In order to determine the number of facets we observe that there are $\sum_{i=1}^k (2^{a_i}-2)$ facets of type $(1)$. Furthermore, there are $2\sum_{i=1}^k (2^{a_i}-2)$ functions $f\colon V\rightarrow \mathbb{Z}$ that are non-constant on one part $A_i$ and constant and equal on all other parts. Thus, there are $2^{\sum_{i=1}^k a_i}-\sum_{i=1}^k (2^{a_i}-2)-2$ functions of type (ii)[a] or [b], where we substracted $-2$ to account for the constant functions which are never facet defining.
\end{proof}

\begin{example}\label{ex:facet}
From Proposition~\ref{prop:facetskpartite} the facets of $P_{K_{a,b}}$ can be easily geometrically described. If $f\colon V\rightarrow \Zz$ is a facet defining hyperplane with $f(v_i)=0$ for all $i$ and $f(w_j)=1$ for all $1\leq j\leq b_1$ and $f(w_j)=-1$ for all $b_1< j\leq b$, then the corresponding facet is
\[
\conv \left( \{e_{w_j}-e_{v_i}\colon j\leq b_1\}\cup \{e_{v_i}-e_{w_j}\colon j> b_1\}\right)
\]
which is isomophic to the convex hull of $\Delta_a\times \Delta_{b_1}\times \{0\}$ and $-\Delta_a\times \{0\}\times -\Delta_{b_2}$ under the isomorphism defined by $e_{v_i}\mapsto -e_{v_i}$ and $e_{w_j}\mapsto e_{w_j}$ for all $i,j$. Here, $b_1+b_2=b$ and $\Delta_a$ denotes the standard simplex on $a$ vertices. In particular, if $b=b_1$ then the facet is isomorphic to a product of two standard simplices.
\end{example}
\begin{remark}
Interestingly, for complete bipartite graph only facets of type (i) from Proposition \ref{prop:facetskpartite} appear, while for complete graph only facets of type (ii)[a]. 
\end{remark}

\subsection{Gr\"obner basis}\label{sec:Groebner}
In this section we provide a Gr\"obner basis for the toric ideal associated to $P_G$ and study the associated induced triangulation. For each edge $e$ of a simple graph $G$ we consider both oriented versions $e^+$ and $e^-$ and associate two variables $x_e,y_e$, one for each possible orientation. 
Since $P_G$ is reflexive and terminal, we can naturally identify $K[\{x_{\alpha}\colon \alpha \in P_G\cap \mathbb{Z}^V\}]$ with $K[\{x_e,y_e\}_{e\in E}\cup \{z\}]$, where $z$ is associated to the origin. 

In order to simplify notation, in the following, for any oriented edge $e$, we denote by $p_e$ the corresponding variable, i.e.~$p_e=x_e$ or $p_e=y_e$ depending on the orientation. We also set $q_e$ to be equal to the variable with the opposite orientation, i.e.~$\{p_e,q_e\}=\{x_e,y_e\}$.

\begin{proposition}\label{prop:GB}
Let $z< x_{e_1}<y_{e_1}<\dots<x_{e_k}<y_{e_k}$ be an order on the edges. Then the following collection of three types of binomials forms a Gr\"obner basis of the toric ideal of $P_G$ with respect to the degrevlex order:
\begin{enumerate}
\item For every $2k$-cycle $C$, with fixed orientation, and any $k$-element subset $I$ of edges of $C$ not containing the smallest edge  
$$\prod_{e\in I}p_e-\prod_{e\in C\setminus I} q_e.$$
\item For every $(2k+1)$-cycle $C$, with fixed orientation, and any $(k+1)$-element subset $I$ of edges of $C$ 
$$\prod_{e\in I}p_e-z\prod_{e\in C\setminus I}q_e.$$
\item For any edge $e$ $$x_ey_e-z^2 \, .$$
\end{enumerate}
The leading monomial is always chosen to have positive sign.
\end{proposition}
\begin{proof}
It is enough to prove that for any binomial $m_1-m_2$ in the toric ideal of $P_G$ on of the monomials $m_1$ or $m_1$ is divisible by the leading monomial of one of the binomials present above (see, e.g., \cite{Stks}). Both monomials can be represented by directed subgraphs $G_1$ and $G_2$ of $G$ in a canonical way, namely $e^+$ and $e^-$ are edges in the graph if and only if $x_e$ or, respectively, $y_e$ are present in the monomial. We may assume that neither $G_1$ nor $G_2$ has a directed cycle of length two since otherwise this monomial is divisible by $x_e y_e$ for some edge $e$. Since $m_1-m_2$ is contained in the toric ideal of $P_G$ the graphs $G_1$ and $G_2$ have the same difference of in and out degrees at every vertex. Let $G_2'$ be the graph obtained from $G_2$ by inverting all edge orientations. Since in $G_1\cup G_2'$ the in degree equals the out degree at every vertex, we can find an Euler path, that is, a directed closed path using every edge exactly once. In particular, we find a cycle $C$ in $G_1\cup G_2'$. Let $a$ be the number of edges of $C\cap G_1$ and $b$ be the number of edges of $C\cap G_2'$. Without loss of generality $a\geq b$. 

First, suppose that $a+b=2k$, $k>1$. If $a>b$ consider the set consisting of the $k$ largest edges of $C\cap G_1$. Then the leading term of the corresponding binomial in (1) divides $m_1$. If $a=b$ we may assume without loss of generality that the smallest edge of $C$ belongs to $G_2$ and proceed as before.
If $a+b=2k+1$, then since $a>b$, the leading term of the binomial in (2) corresponding to $(k+1)$ directed edges in $G_1$ divides $m_1$.
\end{proof}

While the Gr\"obner basis obtained in Proposition~\ref{prop:GB} is in general not \textit{reduced}, an explicit construction of a reduced Gr\"obner basis for $K_{a,b}$ was obtained in \cite{Higashitani2017interlacing}. For edge polytopes a Gr\"obner basis was obtained by Ohsugi and Hibi~\cite{ohsugi1998normal}.
\begin{theorem}[\cite{Higashitani2017interlacing}]\label{thm:groebnerKab}
Let $K_{a,b}$ be a complete bipartite graph with vertex set $\{v_1,\ldots, v_a,w_1,\ldots, w_b\}$ and edge set $\{v_iw_j\colon 1\leq i\leq a, 1\leq j\leq b\}$. Let $e_{ij}$ be a variable associated to the oriented edge $(v_i,w_j)$ and $f_{ij}$ be the variable associated to $(w_j,v_i)$. Let $e_{ij}<e_{i'j'}$ and $f_{ij}<f_{i'j'}$ whenever $i<i'$, or $i=i'$ and $j<j'$ 
and $e_{ij}<f_{i'j'}$ for any choice of $i,j,i',j'$. Then the following are the initial terms of a reduced Gr\"obner basis of the toric ideal associated to $P_{K_{a,b}}$ with respect to the degrevlex order.
\begin{eqnarray}
e_{ij}f_{ij} && \text{ for all }i,j\ \label{gb:1} \\
e_{ij}e_{i'j'} \text{ and } f_{ij}f_{i'j'} && \text{ whenever }i<i'\text{ and }j>j' \ \label{gb:2} \\
e_{ij}f_{i'j} \text{ and } f_{ji}e_{ji'} && \text{ for all } j\neq 1\ \label{gb:3}
\end{eqnarray}
In particular, since all initial terms are quadratic and square free the induced triangulation of $\partial P_{K_{a,b}}$ is unimodular and flag.
\end{theorem}
%
%


\section{Complete bipartite graphs}\label{sec:bipartite}
This section is primarily dedicated to the case of complete bipartite graphs however a few statements generalize to arbitrary graphs. Let $h^\ast _{a,b}(t)$ denote the $h^*$-polynomial of $P_{K_{a+1,b+1}}$. The goal of this section is to study arithmetic properties of $h^\ast _{a,b}(t)$. The first main result is the following simple expression for $h^\ast _{a,b}(t)$. 
\begin{theorem}\label{thm:main1}
For all $a,b\geq 0$
\begin{equation}\label{eq:K_ab1}
h^\ast _{a,b}(t) \ = \ \sum_{i=0}^{\min(a,b)}\binom{2i}{i}\binom{a}{i}\binom{b}{i}t^i(1+t)^{a+b+1-2i}.
\end{equation}
In particular, $h^\ast _{a,b}(t)$ is $\gamma$-positive.
\end{theorem}
In Sections~\ref{sec:goodcolorings}, \ref{sec:facet_desc} and \ref{sec:half}, we employ methods from geometric, enumerative and bijective combinatorics to describe $h^\ast _{a,b}(t)$. These are essential in the proof of Theorem~\ref{thm:main1} in Section~\ref{sec:proof}. In Sections~\ref{sec:roots}, \ref{sec:Nevo} and \ref{sec:recursion} we study arithmetic properties of $h^\ast _{a,b}(t)$.

\subsection{Colorings}\label{sec:goodcolorings}
In this section we give a combinatorial interpretation for the right hand side of \ref{eq:K_ab1}. At the same time we give an alternative expression which serves as a first step towards a proof of Theorem~\ref{thm:main1}.

We consider disjoint sets $A$ and $B$ with $|A|=a$ and $|B|=b$, and  colorings of $A\sqcup B$, that is, maps $c\colon A\sqcup B \rightarrow \{R,G,W,B\}$, where $R,G,W$ and $B$ stand for \textit{red, green, white} and \textit{black}, respectively. Let $g(c):=|c^{-1}(G)|$ denote the number of green elements colored by $c$, and let $r(c),w(c)$ and $b(c)$ be defined analogously. A coloring is called \textbf{good} if the number of red elements in $A$ equals the number of green elements in $B$, and at the same time, the number of green elements in $A$ equals the number of red elements in $B$. In particular, $g(c)=r(c)$ for every good coloring $c$.
\begin{proposition}\label{prop:K_ab2}
\begin{equation}\label{eq:K_ab2}\begin{split}
(1+t)\sum_{ \text{ good }c}t^{g(c)+w(c)}&=\sum_{i=0}^{\min(a,b)}\binom{2i}{i}\binom{a}{i}\binom{b}{i}t^i(1+t)^{a+b+1-2i} \\ &=  (1+t)\sum_{i=0}^a\sum_{j=0}^b\binom{a}{i}\binom{b}{j}\binom{a-i+j}{j}\binom{b+i-j}{i}t^{i+j}. 
\end{split}\end{equation}
where the first sum is taken over all good colorings of $A\sqcup B$.
\end{proposition}
\begin{proof}
We prove that both expressions on the right hand side equal the one on the left hand side.

\textit{First expression:} Let $i$ be the number of elements in $A$ that are either red or green. There are ${a\choose i}$ possibilities of choosing the subset of elements in $A$ that are either green or red. Since we only consider good colorings, there are ${b\choose i}$ possibilities for choosing the subset of all green or red elements in $B$. Among these $2i$ chosen red or green elements in $A\sqcup B$ we have ${2i \choose i }$ possibilities to choose the red elements. It remains to choose $w$ white elements which accounts for ${a+b-2i}\choose w$ possibilities. As $(1+t)^{a+b-2i}=\sum_w {{a+b-2i}\choose w}t^w$ we see that indeed the left hand side equals the first expression on the right hand side.

\textit{Second expression:} Let $X$ be the subset of elements that are either green or white in $A$ and let $Y$ be the subset of green or white elements in $B$. Let $|X|=i$ and $|Y|=j$. Then there are ${a \choose i}{b\choose j}$ possibilities to choose $X$ and $Y$. To determine the red elements in $A$ and the white elements in $B$ we choose a subset $S$ of $(A\setminus X)\cup Y$ of cardinality $j$. We define $S\cap A$ to be the subset of red elements in $A$ and $S\cap B$ to be the subset of white elements in $B$. There are ${{a-i+j}\choose{j}}$ possibilities to choose $S$ and by construction, the number of red elements in $A$ equals the number of green elements in $B$. It remains to choose the red elements in $B$ and, simultaneously, the white elements in $A$ in an analogous way, which accounts for another ${{b-j+i}\choose{i}}$ possibilities.
\end{proof}

\subsection{Triangulation}\label{sec:facet_desc}
Let $\Delta$ be the unimodular triangulation defined by the Gr\"obner basis given in Theorem~\ref{thm:groebnerKab}. 
Every maximal face $\sigma \in \Delta$ in the triangulation corresponds to a directed spanning tree $T(\sigma)$ of the graph in the following way: 
if $e_v -e_w$ is a vertex of the maximal cell, then the directed edge $(w,v)$ is present in $T(\sigma)$. 
Since the non-zero vertices form a maximally linearly independent set, $T(\sigma)$ is a directed spanning tree. 
In the sequel, we think of $K_{a+1,b+1}$ as drawn in the plane in such a way that the vertices lie on two parallel lines and the edges are represented by straight segments connecting its vertices. The vertices on the upper line are labeled by $v_0,v_1,\ldots v_a$ from left to right and the vertices on the lower line are labeled by $w_0,w_1,\ldots,w_b$ from left to right. A spanning tree is called \textbf{planar} if no two of its edges drawn as segments in that way intersect in their interior. Let $T^\uparrow (\sigma)$ denote the edge induced subtree of $T (\sigma)$ consisting of all edges that are directed from the lower towards the upper level of vertices, and, correspondingly, let $T^\downarrow (\sigma)$ denote the edge induced subgraph of downward oriented edges. If there are no upward oriented edges then $T^\uparrow (\sigma):=\{w_0\}$ and, respectively, $T^\downarrow (\sigma):=\{v_0\}$ if there are no downward oriented edges. Let $\mathcal{T}=\{T(\sigma)=(T^\uparrow(\sigma),T^\downarrow (\sigma))\colon \sigma \in \Delta\}$ be the set of all directed spanning trees corresponding to maximal cells in $\Delta$.

The minimal non-faces given by the leading coefficients of the Gr\"obner basis description in Theorem \ref{thm:groebnerKab} correspond to the subgraphs given in Figure \ref{fig:forbidden}. A directed spanning tree of $K_{a+1,b+1}$ is therefore contained in $\mathcal{T}$ if and only if it does not contain any of these subgraphs. The following result characterizes the elements in $\mathcal{T}$.

\begin{proposition}\label{prop:treedescription}
Let $T=(T^\uparrow,T^\downarrow)$ be a directed spanning tree of $K_{a+1,b+1}$. Then $T\in \mathcal{T}$ if and only if
\begin{itemize}
\item[(i)] either $(v_0,w_0)\in T$ or $(w_0,v_0)\in T$, and
\item[(ii)] $T^\uparrow $ and $T^\downarrow $ are planar subtrees, and
\item[(iii)] $T^\uparrow \cap T^\downarrow =\{v_0\}$ or $T^\uparrow \cap T^\downarrow =\{w_0\}$.
\end{itemize}
\end{proposition}
\begin{proof}
Let $T$ be a directed spanning tree that satisfies (i),(ii) and (iii), then it is easily seen that it does not contain any forbidden subgraph given in Figure~\ref{fig:forbidden}. 

For the other direction we assume that $T=T(\sigma)$ for some $\sigma \in \Delta$:

(i) Suppose there is no edge between $v_0$ and $w_0$ in $T(\sigma)$. Since $T(\sigma)$ is a spanning tree there is a unique path $v_0=p_0p_1\ldots p_m=w_0$ and, since $K_{a+1,b+1}$ is bipartite, its length $m$ is odd. By \eqref{gb:3}, every edge $p_ip_{i+1}$ for all even $i$ has the same orientation. In particular, $p_0p_1$ and $p_{m-1}p_m$ have the same orientation. However, since these two edges have to cross this contradicts condition \eqref{gb:2}.

(ii) Planarity of $T^\uparrow (\sigma)$ and $T^\downarrow (\sigma)$ follows directly from condition \eqref{gb:2}. 

(iii) $T^\uparrow(\sigma)$ and $T^\downarrow (\sigma)$ intersect in a vertex since $T(\sigma)$ is spanning. By condition \eqref{gb:3} the only vertices that can possibly be contained in $T^\uparrow (\sigma)\cap T^\downarrow (\sigma)$ are $v_0$ or $w_0$, but not both by condition \eqref{gb:1}.
\end{proof}

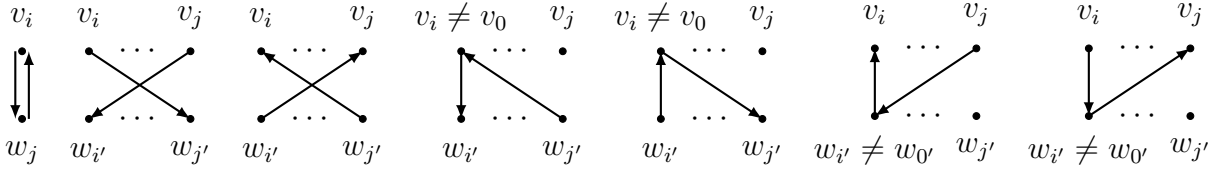
\begin{figure}
\begin{tikzpicture}[scale=0.9]
\filldraw (0,0) circle (0.05cm);

\filldraw (0,1) circle (0.05cm);

\draw node at (0,-0.5) {$w_{j}$};

\draw node at (0,1.5) {$v_{i}$};
\draw[-latex, thick] (0.1,0) -- (0.1,1);
\draw[-latex,thick] (-0.1,1) -- (-0.1,0);
\end{tikzpicture}
\begin{tikzpicture}[scale=0.9]
\filldraw (0,0) circle (0.05cm);
\filldraw (1.5,0) circle (0.05cm);
\filldraw (0,1) circle (0.05cm);
\filldraw (1.5,1) circle (0.05cm);
\draw node at (0.75,0) {$\cdots$};
\draw node at (0.75,1) {$\cdots$};
\draw node at (0,-0.5) {$w_{i'}$};
\draw node at (1.5,-0.5) {$w_{j'}$};
\draw node at (0,1.5) {$v_{i}$};
\draw node at (1.5,1.5) {$v_{j}$};
\draw[-latex, thick] (0,1) -- (1.5,0);
\draw[-latex,thick] (1.5,1) -- (0,0);
\end{tikzpicture}
\begin{tikzpicture}[scale=0.9]
\filldraw (0,0) circle (0.05cm);
\filldraw (1.5,0) circle (0.05cm);
\filldraw (0,1) circle (0.05cm);
\filldraw (1.5,1) circle (0.05cm);
\draw node at (0.75,0) {$\cdots$};
\draw node at (0.75,1) {$\cdots$};
\draw node at (0,-0.5) {$w_{i'}$};
\draw node at (1.5,-0.5) {$w_{j'}$};
\draw node at (0,1.5) {$v_{i}$};
\draw node at (1.5,1.5) {$v_{j}$};
\draw[latex-,thick] (0,1) -- (1.5,0);
\draw[latex-,thick] (1.5,1) -- (0,0);
\end{tikzpicture}
\begin{tikzpicture}[scale=0.9]
\filldraw (0,0) circle (0.05cm);
\filldraw (1.5,0) circle (0.05cm);
\filldraw (0,1) circle (0.05cm);
\filldraw (1.5,1) circle (0.05cm);
\draw node at (0.75,0) {$\cdots$};
\draw node at (0.75,1) {$\cdots$};
\draw node at (0,-0.5) {$w_{i'}$};
\draw node at (1.5,-0.5) {$w_{j'}$};
\draw node at (0,1.5) {$v_{i}\neq v_0$};
\draw node at (1.5,1.5) {$v_{j}$};
\draw[latex-,thick] (0,1) -- (1.5,0);
\draw[-latex,thick] (0,1) -- (0,0);
\end{tikzpicture}
\begin{tikzpicture}[scale=0.9]
\filldraw (0,0) circle (0.05cm);
\filldraw (1.5,0) circle (0.05cm);
\filldraw (0,1) circle (0.05cm);
\filldraw (1.5,1) circle (0.05cm);
\draw node at (0.75,0) {$\cdots$};
\draw node at (0.75,1) {$\cdots$};
\draw node at (0,-0.5) {$w_{i'}$};
\draw node at (1.5,-0.5) {$w_{j'}$};
\draw node at (0,1.5) {$v_{i}\neq v_0$};
\draw node at (1.5,1.5) {$v_{j}$};
\draw[-latex,thick] (0,1) -- (1.5,0);
\draw[latex-,thick] (0,1) -- (0,0);
\end{tikzpicture}
\begin{tikzpicture}[scale=0.9]
\filldraw (0,0) circle (0.05cm);
\filldraw (1.5,0) circle (0.05cm);
\filldraw (0,1) circle (0.05cm);
\filldraw (1.5,1) circle (0.05cm);
\draw node at (0.75,0) {$\cdots$};
\draw node at (0.75,1) {$\cdots$};
\draw node at (0,-0.5) {$w_{i'}\neq w_{0'}$};
\draw node at (1.5,-0.5) {$w_{j'}$};
\draw node at (0,1.5) {$v_{i}$};
\draw node at (1.5,1.5) {$v_{j}$};
\draw[latex-,thick] (0,0) -- (1.5,1);
\draw[latex-,thick] (0,1) -- (0,0);
\end{tikzpicture}
\begin{tikzpicture}[scale=0.9]
\filldraw (0,0) circle (0.05cm);
\filldraw (1.5,0) circle (0.05cm);
\filldraw (0,1) circle (0.05cm);
\filldraw (1.5,1) circle (0.05cm);
\draw node at (0.75,0) {$\cdots$};
\draw node at (0.75,1) {$\cdots$};
\draw node at (0,-0.5) {$w_{i'}\neq w_{0'}$};
\draw node at (1.5,-0.5) {$w_{j'}$};
\draw node at (0,1.5) {$v_{i}$};
\draw node at (1.5,1.5) {$v_{j}$};
\draw[-latex,thick] (0,0) -- (1.5,1);
\draw[-latex,thick] (0,1) -- (0,0);
\end{tikzpicture}
\caption{Forbidden configurations.}\label{fig:forbidden}
\end{figure}
The triangulation $\Delta$ canonically induces a unimodular triangulation of every face of $P_{K_{a+1,b+1}}$. If $F$ is a facet with facet defining linear function $f\colon V\rightarrow \mathbb{Z}$, then a directed tree $T\in \mathcal{T}$ corresponds to a maximal simplex in the triangulation of $F$ if and only if $f(q)-f(p)=1$ for every directed edge $(p,q)\in T$. In Example~\ref{ex:facet} we saw that the facet defined by $f(v_i)=0$ and $f(w_j)=1$ for all $i,j$ is a product of simplices for which the normalized volume can easily be calculated to be ${a+b\choose a}$. Together with Proposition~\ref{prop:treedescription} this yields the following well-known result.
\begin{corollary}\label{cor:spanning}
The number of (undirected) planar spanning trees of $K_{a+1,b+1}$ is ${a+b\choose b}$.
\end{corollary}
On the other hand, more generally, Proposition~\ref{prop:treedescription} allows us to determine the volume of arbitrary facets of $P_{K_{a+1,b+1}}$ by counting spanning trees.
\begin{proposition}
Let $F$ be a facet of $P_{K_{a+1,b+1}}$ with facet defining function $f\colon V \rightarrow \mathbb{Z}$ where $f(v_i)=0$ and $f(w_j)=1$ for $0\leq j\leq b_1$ and $f(w_j)=-1$ for $b_1<j\leq b$. Then $F$ has normalized volume
$$\sum_{i=0}^{a} {b_1+i\choose b_1}{b_2+a-i-1\choose b_2-1}{{a}\choose{i}}.$$
\end{proposition}
\begin{proof}
Let $B_1=\{w_j\colon j\leq b_1\}$ and $B_2=\{w_j\colon j> b_1\}$. Let $\sigma \in \Delta$ be a maximal simplex in the triangulation given by the Gr\"obner basis. Then $\sigma$ is contained in $F$ if and only if all edges of $T^\uparrow (\sigma)$ start in $B_2$ and all edges of $T^\downarrow (\sigma)$ end in $B_1$. By condition (iii) of Proposition~\ref{prop:treedescription}, $T^\uparrow (\sigma)\cap T^\downarrow (\sigma)=\{v_0\}$. By choosing the vertices of $\{v_1,\ldots,v_a\}$ contained in $T^\downarrow (\sigma)$ and counting the number of possible planar spanning trees $T^\uparrow (\sigma)$ and $T^\downarrow (\sigma)$ we obtain the claimed formula as the number of directed spanning trees corresponding to maximal simplices contained in $F$ which equals the normalized volume of $F$.
\end{proof}

\subsection{Half-open triangulation}\label{sec:half}
In this section we will give a combinatorial description of $h^\ast _{a,b}(t)$ by decomposing $P_{K_{a+1,b+1}}$ into half-open unimodular simplices of $\Delta$.

For every $T\in \mathcal{T}$ and every directed edge $\vec{e}$ of $T$, $T\setminus \{ \vec{e} \}$ decomposes into two trees (possibly without edges) corresponding to a codimension $1$ face of $\Delta$ containing the origin. 
Let $U_1$ be the component containing $w_0$ and let $U_2$ denote the other component. By construction, $\vec{e}$ connects $U_1$ and $U_2$. We call $\vec{e}$ {\em ingoing} (into $U_1$) if its orientation goes from a vertex in $U_2$ into a vertex in  $U_1$.

\begin{proposition}\label{prop:ingoing}
Let $h^\ast _{a,b}(t)=\sum_{i=0}^dh^*_it^i$. Then 
\begin{align*}
h_i^*=\sharp\{ T \in \Tm : \text{$T$ has exactly $i$ ingoing edges}\}. 
\end{align*}
\end{proposition}
\begin{proof}
For every $T\in \mathcal{T}$ and every directed edge $\vec{e}$ of $T$, the facet defining hyperplane of the facet corresponding to $T\setminus \{ \vec{e} \}$ is given by the linear function $f\colon \RR^V \rightarrow \RR$ defined by
\begin{align*}
f(v)=\begin{cases}
d+1-|U_1| &\text{ for every $v$ contained in $U_1$}, \\
-|U_1| &\text{ otherwise},
\end{cases}
\end{align*}
where $|U_1|$ denotes the number of vertices in $U_1$. This is easily seen as $f(v)-f(w)=0$ for all edges $(w,v)\in T\setminus \{ \vec{e} \}$.
Let $t \gg 0$ large and let $q \in \mathbb{R}^{V}$ be the point defined by $q_{v_i}=-t^{a+1-i}$ and $q_{w_i}=-t^{a+b+2-i}$. 
Since $t \gg 0$, $q_v \gg q_{w_0}$ for all vertices $v$ in $K_{a+1,b+1}$ and therefore $f(q)<0$. Let $\tilde q=q-\lambda(1,\dots,1)\in \mathbb{R}^{|V|}$, be such that the sum of coordinates of $\tilde q$ equals $0$. In other words $\tilde q$ belongs to the linear span of $P_{K_{a+1,b+1}}$. We have $f(\tilde q)=f(q)$, as $\sum _{v}f(v)=0$.

If $\vec{e}=(v,w)$ is oriented into $U_1$ then $f(w)-f(v)=d+1>0$ and otherwise $f(w)-f(v)=-d-1<0$. That is, in the former case, $\tilde q$ is beyond the facet defining hyperplane. The claim follows now with Proposition~\ref{prop:halfopendecomp}.
\end{proof}

\subsection{Proof of Theorem \ref{thm:main1}}\label{sec:proof}
This section is devoted to the proof of Theorem~\ref{thm:main1}. 

For any $T\in \mathcal{T}$ let $A^\uparrow$ denote the vertices in the upper level and $B^\uparrow$ the vertices in the lower level contained in $T^\uparrow$. Accordingly, $A^\downarrow$ denotes the vertices in the upper level  and $B^\downarrow$ the vertices in the lower level contained in $T^\downarrow$.
\begin{lemma}\label{lem:ingoingnumbers}
Let $T\in \mathcal{T}$. Then the number of ingoing edges in $T$ equals
\begin{itemize}
\item[(i)] $|A^\downarrow|+|B^\uparrow|-2$, if $T^\uparrow\cap T^\downarrow =\{v_0\}$ and $w_0\in T^\uparrow$,
\item[(ii)] $|A^\downarrow|+|B^\uparrow|$, if $T^\uparrow\cap T^\downarrow =\{v_0\}$ and $w_0\in T^\downarrow$,
\item[(iii)] $|A^\downarrow|+|B^\uparrow|-1$, if $T^\uparrow\cap T^\downarrow =\{w_0\}$.
\end{itemize}
In particular, the number of ingoing edges of $T$ only depends on $|A^\downarrow|$ and $|B^\uparrow|$.
\end{lemma}
\begin{proof}
(i) In every vertex in $A^\downarrow \setminus \{v_0\}$ ends exactly one ingoing edge of $T^\downarrow$ and in every vertex of $B^\uparrow \setminus \{w_0\}$ ends exactly one ingoing edge of $T^\uparrow$.

(ii) In every vertex in $A^\downarrow$ ends exactly one ingoing edge of $T^\downarrow$ and in every vertex of $B^\uparrow$ ends exactly one ingoing edge of $T^\uparrow$.

(iii) In every vertex in $A^\downarrow$ ends exactly one ingoing edge of $T^\downarrow$ and in every vertex of $B^\uparrow \setminus  \{w_0\}$ ends exactly one ingoing edge of $T^\uparrow$.
\end{proof}

\begin{proof}[Proof of Theorem~\ref{thm:main1}]
By Proposition~\ref{prop:ingoing}, the $h^\ast$-polynomial of $P_{K_{a+1,b+1}}$ corresponds to the sum of the numbers of all ingoing edges of all $T\in \mathcal{T}$. To determine this number we partition the elements of $\mathcal{T}$ according to the cases (i), (ii) and (iii) in Lemma~\ref{lem:ingoingnumbers}. Let $\mathcal{T}'$, $\mathcal{T}''$ and $\mathcal{T}'''$ denote the subset of $\mathcal{T}$ satisfying condition (i),(ii) or (iii), respectively and for all $T\in \mathcal{T}$ let $\alpha (T)$ denote the number of ingoing edges. 
Since by Lemma~\ref{lem:ingoingnumbers}, the number of ingoing edges only depends on $|B^\uparrow |$ and $|A^\downarrow|$, it suffices to consider all possible choices of $B^\uparrow$ and $A^\downarrow$ (equivalently, $B^\uparrow$ and $A^\uparrow$) and weighting with the corresponding number of pairs of spanning trees on $A^\uparrow  \sqcup B^\uparrow$ and $A^\downarrow  \sqcup B^\downarrow$.
We will use Corollary \ref{cor:spanning} for counting the possible spanning trees $A^\uparrow  \sqcup B^\uparrow$ and $A^\downarrow  \sqcup B^\downarrow$. 

Type (i): Let $j+1:=|B^\uparrow |$ and $i+1:=|A^\uparrow|$, and, equivalently, $b-j=|B^\downarrow|$ and $a+1-i=|A^\downarrow|$. By Lemma~\ref{lem:ingoingnumbers} we obtain
\begin{align}
\sum_{T \in \mathcal{T}'}t^{\alpha(T)}&=\sum_{i=0}^a\sum_{j=0}^{b}\binom{a}{i}\binom{b}{j}\binom{i+j}{j}\binom{a+b-i-j-1}{b-j-1}t^{j+a-i} \nonumber \\
&=\sum_{i=0}^a\sum_{j=0}^{b}\binom{a}{i}\binom{b}{j}\binom{a-i+j}{j}\binom{b+i-j-1}{i}t^{i+j} \label{sum:1}
\end{align}
where the last equation follows from a change of variables $i\mapsto a-i$.

Type (ii): Let $j:=|B^\uparrow |$ and $i+1:=|A^\uparrow|$, and, equivalently, $b-j+1=|B^\downarrow|$ and $a+1-i=|A^\downarrow|$. By Lemma~\ref{lem:ingoingnumbers} we obtain
\begin{align}
\sum_{T \in \mathcal{T}''}t^{\alpha(T)}&=\sum_{i=0}^a\sum_{j=0}^b\binom{a}{i}\binom{b}{j}\binom{i+j-1}{j-1}\binom{a+b-i-j}{b-j}t^{j+a-i+1} \nonumber \\
&=\sum_{i=0}^a\sum_{j=0}^b\binom{a}{i}\binom{b}{j}\binom{a-i+j-1}{j-1}\binom{b+i-j}{i}t^{i+j+1}\label{sum:2}
\end{align}
with the convention that ${k\choose -1}=1$ if $k=-1$ and is equal to zero otherwise.

Type (iii): Let $j+1:=|B^\uparrow |$ and $i:=|A^\uparrow|$, and, equivalently, $b-j+1=|B^\downarrow|$ and $a+1-i=|A^\downarrow|$. By Lemma~\ref{lem:ingoingnumbers} we obtain
\begin{align}
\sum_{T \in \mathcal{T}'''}t^{\alpha(T)}&=\sum_{i=1}^{a+1}\sum_{j=0}^b\binom{a+1}{i}\binom{b}{j}\binom{i+j-1}{i-1}\binom{a+b-i-j}{a-i}t^{j+a-i+1} \nonumber \\
&=\sum_{i=0}^{a}\sum_{j=0}^b\binom{a+1}{i}\binom{b}{j}\binom{a-i+j}{j}\binom{b+i-j-1}{i-1}t^{i+j} \nonumber \\
&=\label{sum:3}\sum_{i=0}^{a}\sum_{j=0}^b\binom{a}{i}\binom{b}{j}\binom{a-i+j}{j}\binom{b+i-j-1}{i-1}t^{i+j}\\ 
&+\sum_{i=0}^{a-1}\sum_{j=0}^b\binom{a}{i}\binom{b}{j}\binom{a-i+j-1}{j}\binom{b+i-j}{i}t^{i+j+1} \label{sum:4}
\end{align}
again with the convention that ${k\choose -1}=1$ if $k=-1$ and is equal to zero otherwise.

Summing up 
\eqref{sum:1}, 
\eqref{sum:3}, 
\eqref{sum:2} and \eqref{sum:4} 
we obtain:
\[
\sum_{T \in \mathcal{T}}t^\alpha (T) = \sum_{i=0}^a\sum_{j=0}^b\binom{a}{i}\binom{b}{j}\binom{a-i+j}{j}\binom{b+i-j}{i}\left(t^{i+j}+t^{i+j+1}\right)
\]
and therefore the proof follows with Proposition~\ref{prop:K_ab2}.
\end{proof}

\subsection{Roots}\label{sec:roots}
In this section we study the roots of the polynomial $h^\ast _{a,b}(t)$. By employing techniques of interlacing polynomials we prove that all roots are real. The following is our main result.
\begin{theorem}\label{thm:main2}
For all $a,b\geq 0$ the polynomial $h^\ast _{a,b}(t)$ has only real roots and
\[
h^\ast _{a,b-1}(t) \preceq h^\ast_{a,b}(t) \, .
\]
\end{theorem}
We are going to prove Theorem~\ref{thm:main2} by investigating the roots of the $\gamma$-polynomial of  $h^\ast _{a,b}(t)$ which equals $\gamma_{a,b} (t):=\sum _{i\geq 0} {a\choose i}{b\choose i}{2i\choose i}t^i$ by Theorem~\ref{thm:main1}. 
\begin{proposition}\label{prop:gammainterlacing}
For all $a,b\geq 1$ 
\[
\gamma _{a,b-1}(t) \preceq \gamma_{a,b}(t) \, .
\]
\end{proposition}
\begin{proof}
By Theorem~\ref{thm:multiplier}, $\{{a\choose i}i!\}_{i\geq 0}$ is a multiplier sequence, since $\sum _{i\geq 0}{a\choose i}i!\frac{x^i}{i!}=(x+1)^a$ is real-rooted. Furthermore, by \cite[Theorem 3.14]{craven2006fox}, also $\{{2i\choose i}\frac{1}{i!}\}_{i\geq 0}$ is a multiplier sequence. Multiplication yields that $\{{2i\choose i}{a\choose i}\}_{i\geq 0}$ is a multiplier sequence. Since $(t+1)^{b-1}$ interlaces $(t+1)^b$ we obtain the result by applying the multiplier sequence $\{{2i\choose i}{a\choose i}\}_{i\geq 0}$ to these two polynomials by Corollary~\ref{cor:preserving interlacing}.
\end{proof}
In certain cases, interlacing of the $\gamma$-polynomials of two palindromic polynomials implies interlacing of the polynomials themselves.
\begin{lemma}\label{lem:interlacinggamma}
Let $f_1(t)$ and $f_2(t)$ be $\gamma$-positive polynomials with $\deg f_2(t)=\deg f_1 (t) +1$ and $f_1(0),f_2(0)\neq 0$.  Let $\gamma _1(t)$ and $\gamma _2(t)$ be the $\gamma$-polynomials of $f_1(t)$ and $f_2(t)$, respectively. If $\gamma _1(t)\preceq \gamma _2(t)$ then $f_1 (t) \preceq f_2(t)$.
\end{lemma}
\begin{proof}
Let $\gamma _i (t) =c_i\prod _j \left(t+a_{i,j}\right)$ for $i=1,2$. Since $f_i$ has only nonnegative coefficients and $f_i (0)\neq 0$ it follows that $a_{i,j}> 0$ for all $j$. By \cite[Proposition 2.1.1]{gal2005real} $\gamma _i(t)$ is the unique polynomial such that
\[
f_i \ = \ (1+t)^{\deg f_i}c_i\gamma _i \left(\frac{t}{(1+t)^2}\right)\ = \ (1+t)^{\deg f_i - 2\deg \gamma _i} c_i\prod _j \left(t+a_{i,j}(1+t)^2\right) \, .
\]
Since $a_{i,j}> 0$, every factor of the form $\left(t+a_{i,j}(1+t)^2\right)$ contributes to two distinct negative real roots of $f_i$, say $b_{i,j,+}>-1>b_{i,j,-}$, which are reciprocals of each other. Calculating these roots explicitely shows that the larger root $b_{i,j,+}$ is monotonically increasing with $a_{i,j}$ and, accordingly, the smaller root $b_{i,j,-}$ is monotonically decreasing. Let $d=\deg \gamma _2$. Then, if $\deg\gamma_1=\deg \gamma _2-1$ we obtain
\begin{equation}\label{eq:real1}
b_{2,1,-}<b_{1,1,-}<\cdots <b_{1,d-1,-}<b_{2,d,-}<-1<b_{2,d,+}<b_{1,d-1,+}<\cdots < b_{1,1,+}<b_{2,1,+} \, .
\end{equation}
Since $\deg f_2- 2\deg \gamma _2 =\deg f_1- 2\deg \gamma _1 -1$, the multiplicity of the zero $-1$ in $f_2$ is by one smaller than the multiplicity in $f_1$ and thus $f_1\preceq f_2$ follows with \eqref{eq:real1}. In the other case, if $d=\deg \gamma _1=\deg \gamma _2$ we have 
\begin{equation}\label{eq:real2}
b_{2,1,-}<b_{1,1,-}<\cdots <b_{2,d,-}<b_{1,d,-}<-1<b_{1,d,+}<b_{2,d,+}<\cdots < b_{1,1,+}<b_{2,1,+} \, .
\end{equation}
In this case, $\deg f_2- 2\deg \gamma _2 =\deg f_1- 2\deg \gamma _1 +1$, that is, the multiplicity of the zero $-1$ in $f_2$ is by one greater than the multiplicity in $f_1$ and thus $f_1\preceq f_2$ follows with \eqref{eq:real2}.
\end{proof}

\begin{proof}[Proof of Theorem~\ref{thm:main2}]
The proof follows from Proposition~\ref{prop:gammainterlacing} and Lemma~\ref{lem:interlacinggamma} applied to $\gamma _{a,b}(t)$.
\end{proof}

\subsection{Flag simplicial complexes and $\gamma$-polynomials}\label{sec:Nevo}
In~\cite{nevo2011gamma}, Nevo and Petersen conjecture the following. 
\begin{con}[{\cite[Conjecture 6.3]{nevo2011gamma}}]\label{conj:nevo}
The $\gamma$-polynomial of any flag triangulation of a simplicial sphere is the $f$-polynomial of a balanced simplicial complex.
\end{con}
Equivalently, the coefficients of the $\gamma$-polynomial satisfies the so-called Frankl--F\"uredi--Kalai inequalities~\cite{frankl1988shadows}. Towards that conjecture we prove the following:
\begin{theorem}
The $\gamma$-polynomial of any flag unimodular triangulation of $\partial P_{K_{a+1,b+1}}$ is the $f$-polynomial of a flag balanced simplicial complex.
\end{theorem}
\begin{proof}
By Theorem~\ref{thm:main1} the $h^\ast$-polynomial of $P_{K_{a+1,b+1}}$ is $\gamma$-positive with $\gamma$-polynomial $\gamma _{a,b}(t)=\sum _{i\geq 0}{2i \choose i}{a\choose i}{b\choose i}t^i$. Since the $h^\ast$-polynomial equals the $h$-polynomial of any unimodular triangulation of the boundary of  $P_{K_{a+1,b+1}}$, the proof follows from Proposition~\ref{thm:Nevo} below.
\end{proof}

\begin{proposition}\label{thm:Nevo}
For all $a,b\geq 1$ the polynomial
\[
\sum _{i\geq 0}{2i \choose i}{a\choose i}{b\choose i}t^i\]
is the $f$-polynomial of a flag balanced simplicial complex. 
\end{proposition}
\begin{proof}
Without loss of generality we may assume that $a\leq b$. Let $X=\{x_{i,j}: 1\leq i\leq a, 1\leq j\leq b\}$, $Y=\{y_{i,j}: 1\leq i\leq a, 1\leq j\leq b \}$ and let $V= X \cup Y$ be the set of vertices of the simplicial complex $\Delta$ that we define by the following set of minimal non-faces:
\begin{itemize}
\item $\{x_{i,j},  x_{i',j'}\}$ for $i\leq i'$ and $j\geq j'$,
\item $\{y_{i,j} ,y_{i',j'}\}$ for $i\leq i'$ and $j\geq j'$, and
\item $\{x_{i,j},y_{i',j}\}$ and $\{x_{i,j},y_{i,j'}\}$ for any $i,j,i',j'$. 
\end{itemize}
By definition, $\Delta$ is a flag simplicial complex, and of dimension $(a-1)$ since $\{x_{1,1},\dots,x_{a,a}\}$ defines a simplex. Assigning to all $x_{i,j}$ and $y_{i,j}$ the color $i$ moreover shows that $\Delta$ is a balanced simplicial complex. 

It remains to prove that the $f$-vector of $\Delta$ is as predicted in the proposition. To see that let $A=\{1,\dots,a\}$ and $B=\{1,\dots, b\}$.
Reminiscent of the proof of Proposition~\ref{prop:K_ab2} we define a partial coloring with colors red and green of the set $A \sqcup B$ to be \textit{good} if there are as many green elements in $A$ as there are red elements in $B$ and at the same time there are as many green elements in $B$ as there are red elements in $A$.
Just as in Proposition \ref{prop:K_ab2} we see that the polynomial we are interested in is of the form
$$\sum_{\text{good }c}t^{g(c)} \ = \ \sum _{i\geq 0}{2i \choose i}{a\choose i}{b\choose i}t^i \, $$
where the sum is over all good colorings $c$ and $g(c)$ is the number of green elements. To finish the proof we establish a bijection between simplices in $\Delta$ with $k$ vertices and good colorings of $A\sqcup B$ with $k$ green elements.

For the first direction, for every simplex $\sigma$ of $\Delta$ let
\[
X_A (\sigma)=\{i\colon x_{i,j}\in \sigma\} \, , X_B (\sigma)=\{j\colon x_{i,j}\in \sigma\} \, , Y_A (\sigma)=\{i\colon y_{i,j}\in \sigma\} \, , Y_B (\sigma)=\{j\colon y_{i,j}\in \sigma\} \, .
\]
Then, by definition of $\Delta$, 
\[
X_A(\sigma)\cap Y_A(\sigma)=\emptyset \text{ and } X_B(\sigma)\cap Y_B(\sigma)=\emptyset \, .
\]
We define a coloring of $A\sqcup B$ by setting $X_A\subseteq A$ to be the green elements and $Y_A\subseteq A$ to red elements in $A$, and $X_B\subseteq B$ the red elements and $Y_B\subseteq B$ the green elements in $B$. Since every $x_{i,j}$ and $y_{i',j'}$ in $\sigma$ contributes to precisely one green element and one red element, this defines a good coloring and $|\sigma|$ equals the total number of green elements in $A\sqcup B$.  

For the inverse map, consider a good coloring of $A\sqcup B$ with $k$ green elements and let
\begin{itemize}
\item $i_1< i_2< \dots< i_s$ be the green elements of $A$; 
\item $j_1<j_2<\dots<j_t$ be the red elements of $B$; 
\item $i_1'<\dots<i_{k-s}'$ be the red elements of $A$; 
\item $j_1'<\dots<j_{k-t}'$ be the green elements of $B$. 
\end{itemize}
Since the coloring is good, we have $s=t$ and we may associate a simplex with vertices $x_{i_1,j_1}, x_{i_2,j_2},\dots,x_{i_s,j_s}$ and $y_{i_1',j_1'},\dots,y_{i_{k-s}',j_{k-s}'}$. By definition, this simplex belongs to $\Delta$, has $k$ vertices and both maps are easily seen to be inverses of each other, which finishes the proof.
\end{proof}

\subsection{A recursive formula}\label{sec:recursion}
In~\cite{Higashitani2017interlacing} recursive formulas for $h^\ast_{a,b}$ were given for any fixed $a\leq 2$. These formulas played a fundamental role in the study of the roots of the Ehrhart polynomial of $P_{K_{a+1,b+1}}$.  Here we present a general formula for arbitrary $a$ and $b$.
\begin{proposition}\label{prop:recur}
For all $a,b\geq 1$
$$(b-a)h^*_{a,b}=(1+t)(bh^*_{a,b-1}-ah^*_{a-1,b}) \, .$$
\end{proposition}
\begin{proof}
By Theorem~\ref{thm:main1},
$$b\left(h^*_{a,b}-(1+t)h^*_{a,b-1}\right)=\sum _{i\geq 0}{{2i}\choose i}{a\choose i}t^i(1+t)^{a+b+1-2i}b\left({b\choose i}-{{b-1}\choose i}\right).$$
As $b\left({b\choose i}-{{b-1}\choose i}\right)=i{b\choose i}$ the above polynomial equals
$$\sum_{i\geq 0} i{{2i}\choose i}{a\choose i}{b\choose i}t^i(1+t)^{a+b+1-2i}.$$
However, by the same argument, this is also equal to
$$a\left(h^*_{a,b}-(1+t)h^*_{a-1,b}\right) \, ,$$
which proves the proposition.
\end{proof}

\textbf{Acknowledgements:} The authors would like to thank the Mathematisches Forschungsinstitut Oberwolfach for hosting the Mini-Workshop
``Lattice polytopes:
methods, advances and applications''
in fall 2017 during which this project evolved.
They also would like to thank Petter Br\"anden, Christian Haase and Eran Nevo for helpful comments. 
Akihiro Higashitani was partially supported by JSPS Grant-in-Aid for Young Scientists (B) $\sharp$17K14177. 
Katharina Jochemko was supported by the Knut and Alice Wallenberg foundation. 
Mateusz Michałek was supported by the Polish National Science Centre grant no. 2015/19/D/ST1/01180. 


 
\bibliographystyle{plain}
\bibliography{Xbib}



\end{document}